\documentclass{article}
\usepackage{fullpage}
\usepackage{amsthm}
\newtheorem{theorem}{Theorem}
\newtheorem{lemma}[theorem]{Lemma}
\newtheorem{proposition}[theorem]{Proposition}
\newenvironment{keywords}{\begin{center}Keywords:}{\end{center}}
\newcommand{\email}[1]{Email: #1}

\usepackage{amsmath,amssymb,amsfonts}
\usepackage{tikz}

\newcommand{\Z}{\mathbb{Z}}
\newcommand{\R}{\mathbb{R}}
\newcommand{\Q}{\mathbb{Q}}

\newcommand{\card}{\#}
\newcommand{\st}{\text{s.t.}}
\newcommand{\E}{\mathbb{E}}

\newcommand{\B}{{\cal B}}
\newcommand{\NB}{{\cal N}}
\renewcommand{\d}{\,\textnormal{d}}
\newcommand{\medcup}{\scalebox{1.5}{$\cup$}}
\newcommand{\medcap}{\scalebox{1.5}{$\cap$}}
\DeclareMathOperator{\vol}{vol}

\title{On the complexity of two-stage stochastic linear\\ optimization with small second stage}

\author{Christoph~Buchheim\thanks{Fakult\"at f\"ur Mathematik, TU Dortmund, Germany. \email{christoph.buchheim@tu-dortmund.de}}}

\begin{document}

\maketitle

\begin{abstract}
  Two-stage stochastic linear optimization is known to be \card P-hard
  when all involved random variables are independently and uniformly
  distributed over intervals, even with fixed recourse. We show that
  this problem is in fact \card P-hard in the strong sense. More
  surprisingly, \card P-hardness persists even when the random vector
  is one-dimensional, i.e., uniformly distributed over a single
  interval. To obtain this result, we show that computing the area of
  a two-dimensional polytope given by a compact extended formulation
  is \card P-hard. Furthermore, for a problem in standard form, we
  prove strong \card P-hardness in case the number of second-stage
  constraints is fixed, while fixing the number of second-stage
  variables leads to a weakly \card P-hard problem. Finally, if both
  the dimension of the random vector and the number of second-stage
  constraints are fixed, the problem turns out to be tractable.
\end{abstract}


\section{Introduction}

In two-stage stochastic optimization, some of the problem parameters
may depend on a random vector. The decision maker has to take some of
her decisions before the random variables are realized (the
\emph{first stage} or \emph{here-and-now} decisions), while the
remaining decisions can be taken afterwards (the \emph{second stage},
\emph{wait-and-see}, or \emph{recourse} decisions). The goal is to
optimize the expected objective value. The latter depends on the
probability distribution of the random parameters, which we assume to
be part of the input. In case of a linear optimization problem, the
two-stage stochastic counterpart can be stated as
\begin{equation}\label{prob:sslp}\tag{SSLP}
\begin{array}{rl}
  \min \quad & c^\top x +\E_\xi[Q(x,\xi)]\\[0.75ex]
  \st \quad & Ax\le b\;,
\end{array}
\end{equation}
where the second-stage objective value~$Q(x,\xi)$ depends on the
first-stage decision~$x$ and on a random variable~$\xi$, namely,
\begin{align*}
  Q(x,\xi) \; = \; \min \quad & q_\xi^\top y\\
  \st \quad & T_\xi x+W y=h_\xi,\quad y\ge 0\;.
\end{align*}
Here the data $c\in\R^{n_1}$, $A\in\R^{m_1\times n_1}$,
and~$b\in\R^{m_1}$ describe the first-stage objective and constraints,
where we assume for sake of simplicity that~$\{x\in\R^{n_1}\mid Ax\le
b\}$ is non-empty and bounded.  Moreover,~$q_\xi\in\R^{n_2}$,
$T_\xi\in\R^{m_2\times n_1}$, $W\in\R^{m_2\times n_2}$,
and~$h_\xi\in\R^{m_2}$ define the second-stage linear optimization
problem over the variables~$y\in\R^{n_2}$, where we assume that~$W$
has full row rank. Finally, we assume that~$q_\xi$, $T_\xi$,
and~$h_\xi$ depend on the random variable~$\xi$ in an affine-linear
way; this assumption is common in stochastic optimization and crucial
for our positive results.  Moreover, in this note, we restrict
ourselves to the case of fixed recourse, meaning that~$W$ does not
depend on~$\xi$. For general results on stochastic linear
optimization, we refer to the textbook~\cite{Birge11}.

If the random variable~$\xi$ has finite support, it is easy to see
that~\eqref{prob:sslp} can be reformulated as an equivalent single
stage linear optimization problem. For this, we can essentially
introduce a copy of the second stage for each possible realization
of~$\xi$. In particular, problem~\eqref{prob:sslp} can be solved
efficiently if the support of~$\xi$ is given explicitly as part of the
input. This is not possible, however, if the distribution of~$\xi$ is
continuous or if the support is finite but exponential in the input.
In fact, the complexity picture changes drastically if~$\xi$ follows a
continuous uniform distribution, even if all entries of~$\xi$ are
independently distributed. In the remainder of this note, we
assume~$\xi\sim{\cal U}[l,u]$ with~$l,u\in\R^d$, $l<u$. It was shown
by Hanasusanto et al.\ \cite{Hanasusanto16} that~\eqref{prob:sslp} is
\card P-hard in this setting, even if only the objective
vector~$q_\xi$ is random. Since the objective in~\eqref{prob:sslp} is
convex in~$x$, the difficulty of solving~\eqref{prob:sslp} essentially
stems from the difficulty of computing objective values, i.e., of
computing~$\E_\xi[Q(x,\xi)]$ for a given first-stage
decision~$x$. This result even holds for~$n_1=1$, i.e., in case of a
single first-stage variable.

For the fixed recourse setting we consider here, the proof presented
in~\cite{Hanasusanto16} requires both an unbounded number of
second-stage variables and an unbounded dimension of the random
vector~$\xi$. It is thus a natural question to ask which complexity
results persist if one or both of these parameters are fixed. One
could argue that in practice the sample space is often
low-dimensional, even if many coefficients in the second-stage problem
are affected by the uncertainty. For instance, the only relevant
random variables in an agricultural application may be the amount of
rain and sunshine during a certain period. Moreover, using principal
component analysis, the dimension of the random vector can often be
reduced significantly without losing too much in terms of solution
quality. This motivated our investigation of the complexity
of~\eqref{prob:sslp} with fixed dimension~$d$.

It directly follows from the results in~\cite{Hanasusanto16} that
\card P-hardness still holds if the second stage contains only one
constraint and two variables. We will show that, in fact,
Problem~\eqref{prob:sslp} turns out to be \card P-hard even when
restricted to the case~$d=1$, i.e., to a single random variable
(Section~\ref{sec:fixed_d}), and \emph{strongly} \card P-hard for
$m_2=1$ (Section~\ref{sec:fixed_mn}). However, if both~$m_2$ and~$d$
are fixed, the problem turns tractable (Section~\ref{sec:fixed_both}).

To the best of our knowledge, stochastic programming with bounded
second stage size has not been investigated in our setting.  Forcier
et al.~\cite{Forcier24} take a complementary approach to multi-stage
stochastic linear programs with arbitrary cost distributions and
finitely supported constraint random variables. They show that the
expected cost-to-go functions are polyhedral and affine on the cells
of a distribution-independent chamber complex, and establish an exact
quantization: the problem is equivalent to one with a finitely
supported cost distribution. For fixed recourse dimensions and fixed
number of constraint scenarios, this yields polynomial-time
solvability. However, this is not directly related to our results.

\section{Fixed dimension of the random vector}\label{sec:fixed_d}

We first discuss the special case of Problem~\eqref{prob:sslp} where
the dimension~$d$ of the random vector is fixed. Our aim is to reduce
\card IS, the \card P-complete counting version of the independent set
problem, to the restriction of~\eqref{prob:sslp} with~$d=1$. This will
imply that the latter restriction is still \card P-hard. To this end,
we use a construction that is inspired by the extended formulation of
a regular polygon presented by Ben-Tal and
Nemirovski~\cite{BenTal01}. For the following, let~$Q_k$ denote the
regular $k$-gon centered in the origin with vertices
\[
v_i:=\sec(\tfrac\pi k)(\cos((2i+1)\tfrac\pi k),\sin((2i+1)\tfrac\pi
k))^\top,\quad i=0,\dots,k-1\;.\] Then $\|v_i\|_2=\sec(\tfrac\pi k)$
for all~$i=0,\dots,k-1$ and~$\vol(Q_{k})=k\tan(\tfrac\pi k)$.  The
center point of the facet of~$Q_k$ connecting~$v_i$ and~$v_{i+1}$
(setting~$v_{k}:=v_0$) is denoted by
\[
w_i:=\tfrac 12(v_i+v_{i+1})=(\cos((2i+2)\tfrac\pi
k),\sin((2i+2)\tfrac\pi k))^\top\;,\] it has length~$\|w_i\|_2=1$ and
agrees with the normal vector of the corresponding facet for
all~$i=0,\dots,k-1$. In particular, $Q_k$ has inradius 1.

Now to start the construction, let a simple undirected graph~$G=(V,E)$
be given, where we assume~$V=\{1,\dots,n\}$ with~$n\ge 2$. Consider
the polyhedron
\begin{eqnarray*}
  P_n ~ := ~ & \big\{ & (y,z)\in\R^{n+1}\times\R^{n+1}\mid\\
  && \; \left.\begin{array}{l}
    \; y_{i+1} = \cos(\tfrac\pi{2^{i-1}})y_{i}+\sin(\tfrac\pi{2^{i-1}})z_{i}\\[1ex]
    \; z_{i+1} \ge -\sin(\tfrac\pi{2^{i-1}})y_{i}+\cos(\tfrac\pi{2^{i-1}})z_{i}\\[1ex]
    \; z_{i+1} \ge +\sin(\tfrac\pi{2^{i-1}})y_{i}-\cos(\tfrac\pi{2^{i-1}})z_{i}
  \end{array} ~ \right\} ~ \forall i=1,\dots,n\\
  &&\quad\,  y_{n+1} \le 1, ~ \cos(\tfrac\pi{2^{n-1}})y_{n+1}+\sin(\tfrac\pi{2^{n-1}})z_{n+1}\le 1\big\}\;.
\end{eqnarray*}
This formulation contains the matrix
\[
R_i:=\begin{pmatrix} \cos(\tfrac\pi{2^{i-1}}) &
\sin(\tfrac\pi{2^{i-1}})\\ -\sin(\tfrac\pi{2^{i-1}}) &
\cos(\tfrac\pi{2^{i-1}})\\
\end{pmatrix}
\]
which describes the clockwise rotation in~$\R^2$ around the origin
by~$\tfrac\pi{2^{i-1}}$ radians, i.e., by one $2^i$-th of a full
rotation. Moreover, we will denote by~$R$ the reflection in~$\R^2$
with respect to the first axis, which flips the sign of the second
entry. The geometric idea behind the construction in~\cite{BenTal01}
is captured in the following lemma.

\begin{lemma}\label{lem:basic}
  For each vertex~$v_\ell$ of~$Q_{2^n}$, there exists a
  unique~$b^\ell\in\{0,1\}^n$ with
  \[
  R^{b^\ell_n}R_n\cdots R^{b^\ell_1}R_1 v_\ell=v_0\;.
  \]
\end{lemma}
\begin{proof}
  To show existence, we define~$b^\ell$ recursively for~$i=1,\dots,n$,
  setting~$b^\ell_i:=1$ if and only if
  $$R_{i}R^{b^\ell_{i-1}}R_{i-1}\cdots R^{b^\ell_1}R_1 v_\ell$$
  has a
  negative second entry. Then it follows by induction that the angle
  of~$R^{b^\ell_{i}}R_{i}\cdots R^{b^\ell_1}R_1 v_\ell$ lies in the
  range~$[0,\tfrac\pi{2^{i-1}}]$. In particular, the final
  vector~$v=R^{b^\ell_{n}}R_{n}\cdots R^{b^\ell_1}R_1 v_\ell$ has an
  angle in~$[0,\tfrac\pi{2^{n-1}}]$. Moreover,~$v=v_{\ell'}$ for
  some~$\ell'\in\{0,\dots,2^n-1\}$, since all involved rotations and
  reflections permute the vertices of~$Q_{2^n}$. Thus~$v=v_0$, the
  only vertex with angle in~$[0,\tfrac\pi{2^{n-1}}]$.

  To show uniqueness, let~$b\in\{0,1\}^n\setminus\{b^\ell\}$ and
  let~$i$ be minimal with~$b_i\neq b^\ell_i$. Then~$R^{b_i}R_{i}\cdots
  R^{b_1}R_1 v_\ell$ has a non-positive second entry and an angle
  in~$[\tfrac\pi{2^{i-1}},2\pi]$. It follows by induction
  that~$R^{b_j}R_{j}\cdots R^{b_1}R_1 v_\ell$ has an angle
  in~$[\tfrac\pi{2^{j-1}},2\pi]$ for all~$j\ge i$,
  hence~$R^{b_n}R_{n}\cdots R^{b_1}R_1 v_\ell$ cannot agree
  with~$v_0$, which has an angle of~$\tfrac\pi{2^n}$.
\end{proof}
By Lemma~\ref{lem:basic}, the vertices~$v_\ell$ of~$Q_{2^n}$ are in
one-to-one correspondence with the binary
vectors~$b_\ell\in\{0,1\}^n$; see Figure~\ref{fig:vertices} for an
illustration. In particular, they correspond to the subsets of~$V$
via~$S_\ell:=\{i\in V\mid b^\ell_i=0\}$. Note that the connection
between~$\ell$ and~$b^\ell$ is different from the standard binary
expansion of~$\ell$.
\begin{figure}
  \begin{center}
    \begin{tikzpicture}[scale=1.1]
      \def\k{4}
      \foreach \i in {1,...,\k} {
        \pgfmathsetmacro{\angle}{(2*\i-1)*180/\k}
        \pgfmathsetmacro{\radius}{1/cos(180/\k)}
        \coordinate (v\i) at ({\radius*cos(\angle)}, {\radius*sin(\angle)});
      }
      \draw[thin,gray!50!white] (0,0) circle (1);
      \draw[thick] (v1) \foreach \i in {2,...,\k} { -- (v\i) } -- cycle;
      \foreach \i in {1,...,\k} {
        \pgfmathsetmacro{\angle}{(2*\i-1)*180/\k}
        \ifcase\i
        \or
        \def\label{10}
        \or
        \def\label{11}
        \or
        \def\label{01}
        \or
        \def\label{00}
        \fi
        \def\slabel{\scriptsize{\label}};
        \node[shift={(\angle:9pt)}] at (v\i) {\slabel};
      }
      \draw[dashed] (-1.5,0)--(1.5,0);
    \end{tikzpicture}
    \qquad
    \begin{tikzpicture}[scale=1.1]
      \def\k{8}
      \foreach \i in {1,...,\k} {
        \pgfmathsetmacro{\angle}{(2*\i-1)*180/\k}
        \pgfmathsetmacro{\radius}{1/cos(180/\k)}
        \coordinate (v\i) at ({\radius*cos(\angle)}, {\radius*sin(\angle)});
      }
      \draw[thin,gray!50!white] (0,0) circle (1);
      \draw[thick] (v1) \foreach \i in {2,...,\k} { -- (v\i) } -- cycle;
      \foreach \i in {1,...,\k} {
        \pgfmathsetmacro{\angle}{(2*\i-1)*180/\k}
        \ifcase\i
        \or
        \def\label{100}
        \or
        \def\label{101}
        \or
        \def\label{111}
        \or
        \def\label{110}
        \or
        \def\label{010}
        \or
        \def\label{011}
        \or
        \def\label{001}
        \or
        \def\label{000}
        \fi
        \def\slabel{\scriptsize{\label}};
        \node[shift={(\angle:7.5pt)}] at (v\i) {\slabel};
      }
      \draw[dashed] (-1.5,0)--(1.5,0);
    \end{tikzpicture}
  \end{center}
  \caption{The regular $2^n$-gon $Q_{2^n}$ for $n=2,3$ and the binary
    vectors~$b^\ell$ assigned to its vertices.}\label{fig:vertices}
\end{figure}
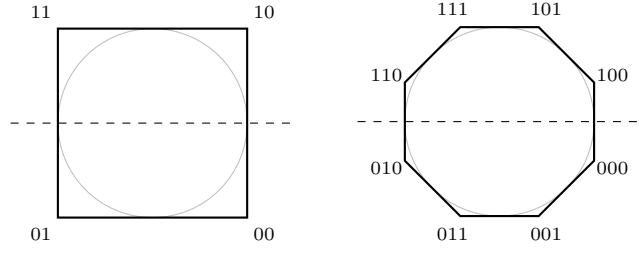
We also remark that~$R^{b^\ell_n}R_n\cdots R^{b^\ell_1}R_1$ maps any
vector with an angle in~$[\ell\tfrac \pi{2^{n-1}},(\ell+1)\tfrac
  \pi{2^{n-1}}]$, i.e., between the angles of~$w_{\ell-1}$
and~$w_\ell$, to a vector with angle in~$[0,\tfrac \pi{2^{n-1}}]$.

The following result is shown in~\cite[Prop.~2.1]{BenTal01}, but we
give a slightly different proof here, as its core idea is essential
for motivating the rest of the construction.
\begin{proposition}\label{prop:proj}
  The projection $\bar P_n$ of~$P_n$
  to~$(y_1,z_1)$ agrees with~$Q_{2^n}$.
\end{proposition}
\begin{proof}
  To show~$Q_{2^n}\subseteq \bar P_n$, it suffices to consider the
  vertices~$v_\ell$ of~$Q_{2^n}$. Then
  defining~$(y_1,z_1)^\top:=v_\ell$
  and~$(y_{i+1},z_{i+1})^\top:=R^{b^\ell_i}R_i(y_i,z_i)^\top$,
  for~$i=1,\dots,n$, yields a vector~$(y,z)$ that satisfies the
  recursive constraints in~$P_n$ by construction. Furthermore, we
  obtain $(y_{n+1},z_{n+1})^\top=v_0=(1,\tan(\tfrac\pi{2^n}))^\top$
  from Lemma~\ref{lem:basic}, which satisfies the remaining two
  constraints.

  For the other inclusion, let~$(y,z)\in P_n$. We show
  that~$(y_1,z_1)$ then satisfies all facet-defining inequalities
  for~$Q_{2^n}$, i.e.,~$w_\ell^\top (y_1,z_1)^\top\le 1$ for
  each~$\ell$.  As mentioned above, $R^{b^\ell_i}R_i\cdots
  R^{b^\ell_1}R_1w_\ell$ has a non-negative second entry for
  all~$i=1,\dots,n$. In particular, since all involved rotations and
  reflections are orthogonal, we obtain
  \begin{eqnarray*}
    w_\ell^\top(y_1,z_1)^\top
    & = & (R^{b^\ell_1}R_1w_\ell)^\top R^{b^\ell_1}R_1(y_1,z_1)^\top\\
    & \le & (R^{b^\ell_1}R_1w_\ell)^\top (y_2,z_2)^\top\\
    & = & (R^{b^\ell_2}R_2R^{b^\ell_1}R_1w_\ell)^\top R^{b^\ell_2}R_2(y_2,z_2)^\top\\
    & \le & (R^{b^\ell_2}R_2R^{b^\ell_1}R_1w_\ell)^\top (y_3,z_3)^\top\\
    & \small\vdots & \\
    & \le & (R^{b^\ell_n}R_n\cdots R^{b^\ell_1}R_1w_\ell)^\top (y_{n+1},z_{n+1})^\top\\
    & \le &1\;,
  \end{eqnarray*}
  where the last inequality uses~$R^{b^\ell_n}R_n\cdots
  R^{b^\ell_1}R_1w_\ell\in\{w_0,w_{2^n-1}\}$ and the two final
  constraints in the definition of~$P_n$.
\end{proof}
In the recursive construction of~$(y,z)$ from~$v_\ell$ applied in the
proof of Proposition~\ref{prop:proj}, the
vector~$R_i(y_{i},z_{i})^\top$ is reflected if and only
if~$b^\ell_i=1$. We will exploit this for cutting off from~$\bar P_n$
exactly those vertices corresponding to dependent sets in~$G$. To this
end, we define a further polyhedron for each edge~$e\in E$ as follows:
\begin{eqnarray*}
  P^e_n:= & \big\{ & (y,z)\in\R^{n+1}\times\R^{n+1}\mid\\
  && \left.\begin{array}{l}
    \,\,y_{i+1} = \cos(\tfrac\pi{2^{i-1}})y_{i}+\sin(\tfrac\pi{2^{i-1}})z_{i}\\[1ex]
    \,\,z_{i+1} = -\sin(\tfrac\pi{2^{i-1}})y_{i}+\cos(\tfrac\pi{2^{i-1}})z_{i} \hspace*{4.25ex} \text{ if~$i\in e$}\\[1ex]
    \hspace*{-0.66ex}\left.\begin{array}{l}
    z_{i+1} \ge -\sin(\tfrac\pi{2^{i-1}})y_{i}+\cos(\tfrac\pi{2^{i-1}})z_{i}\\[1ex]
    z_{i+1} \ge +\sin(\tfrac\pi{2^{i-1}})y_{i}-\cos(\tfrac\pi{2^{i-1}})z_{i}
    \end{array}
    ~ \right\} \text{ if~$i\not\in e$}
  \end{array}\right\} ~ \forall i=1,\dots,n\\
  &&\;\;\;\, \cos(\tfrac\pi{2^n})y_{n+1}+\sin(\tfrac\pi{2^n})z_{n+1}\le \cos(\tfrac\pi{2^n})
  \big\}
\end{eqnarray*}
Finally, we denote by~$\bar P^e_n$ the projection of~$P^e_n$
to~$(y_1,z_1)$ and define
\[
\bar P_G:=\bar P_n\cap\underset{e\in E}{\medcap}\bar P^e_n\;.
\]
Note that the intersection is taken \emph{after} projection. In other
words, the variables~$y_i,z_i$ for~$i=2,\dots,n+1$ are different in
each extended formulation.

In summary, we have built an extended formulation of the
two-dimensional polytope~$\bar P_G$ that, up to the irrational
coefficients, can be constructed in polynomial time from the
graph~$G$. We will first show that we can derive the number of
independent sets of~$G$ analytically from the area of~$\bar P_G$ and
then discuss that this is also possible computationally, after
approximating the irrational coefficients sufficiently well.
\begin{lemma}\label{lem:main}
  For each~$e\in E$, the intersection $\bar
  P_n\cap\bar P^e_n$ agrees with the closure of
  \[
  Q_{2^n}\setminus\underset{\ell:\,e\subseteq
    S_\ell}{\medcup}\Delta_\ell\;,
  \]
  where $\Delta_\ell$ is the triangle spanned by~$w_{\ell-1}$,
  $v_\ell$, and~$w_{\ell}$.
\end{lemma}
\begin{proof}
  In view of Proposition~\ref{prop:proj} and the definition
  of~$Q_{2^n}$, it suffices to verify the following three statements:
  \begin{itemize}
  \item[(a)] For all~$\ell=0,\dots,2^n-1$, we have~$w_\ell\in\bar
    P^e_n$.
  \item[(b)] If~$e\not\subseteq S_\ell$, then $v_\ell\in \bar P^e_n$.
  \item[(c)] If~$e\subseteq S_\ell$, the maximum of~$v_\ell^\top
    (y_1,z_1)^\top$ over~$\bar P_n^e$ is attained at~$w_{\ell-1}$
    and~$w_{\ell}$.
  \end{itemize}
  To show (a), we again define~$(y,z)\in P_n^e$ recursively. Starting
  with~$(y_1,z_1)^\top:=w_\ell$, we now set
  $(y_{i+1},z_{i+1})^\top:=R_i(y_i,z_i)^\top$ if~$i\in e$ or if the
  second entry of~$R_i(y_i,z_i)^\top$ is non-negative,
  and~$(y_{i+1},z_{i+1})^\top:=RR_i(y_i,z_i)^\top$ in any other
  case. Then the resulting vector~$(y,z)$ satisfies all recursive
  constraints and~$(y_{n+1},z_{n+1})$ agrees with~$w_{\ell'}$ for
  some~$\ell'\in\{0,\dots,2^n-1\}$, thus it also satisfies the last
  constraint in~$P_n^e$.

  For (b), we use the same construction as in (a), now starting
  with~$(y_1,z_1)^\top=v_\ell$. Since~$e\not\subseteq S_\ell$, we must
  obtain~$(y_{n+1},z_{n+1})^\top=R^{b_n}R_n\cdots R^{b_1}R_1v_\ell$
  with~$b\neq b^\ell$, so that~$(y_{n+1},z_{n+1})^\top=v_{\ell'}$
  with~$\ell'\neq 0$. It follows again that~$(y_{n+1},z_{n+1})$
  satisfies the last constraint in~$P_n^e$.

  Finally, to show (c), let~$e\subseteq S_\ell$. Then we
  obtain~$v_\ell^\top(y_1,z_1)\le v_0^\top(y_{n+1},z_{n+1})$ exactly
  as in the second direction of the proof of
  Proposition~\ref{prop:proj}, except that the $i$-th and $j$-th
  inequalities now turn into equations for~$e=\{i,j\}$. From the last
  constraint, we thus derive~$v_\ell^\top(y_1,z_1)\le 1$, and this
  bound is attained at~$w_{\ell-1}$ and~$w_\ell$.
\end{proof}
\noindent The statement of Lemma~\ref{lem:main} is illustrated in
Figure~\ref{fig:main}.
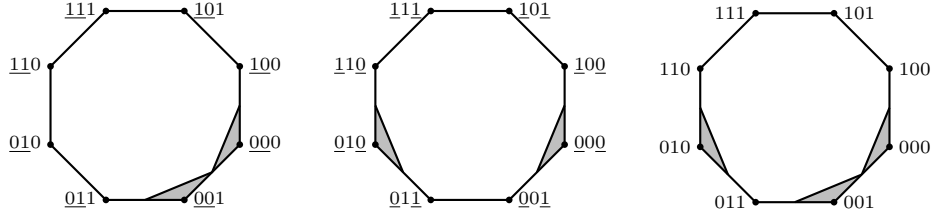
\begin{figure}
  \begin{center}
    \begin{tikzpicture}[scale=1.2]
      \def\k{8}
      \foreach \i in {1,...,\k} {
        \pgfmathsetmacro{\angle}{(2*\i-1)*180/\k}
        \pgfmathsetmacro{\radius}{1/cos(180/\k)}
        \coordinate (v\i) at ({\radius*cos(\angle)}, {\radius*sin(\angle)});
        \coordinate (w\i) at ({cos(\angle+180/\k)}, {sin(\angle+180/\k)});
      }

      \fill[gray!50] (v7) -- (w7) -- (w6) -- cycle;
      \draw[thick] (w7) -- (w6);
      \fill[gray!50] (v8) -- (w8) -- (w7) -- cycle;
      \draw[thick] (w8) -- (w7);

      \draw[thick] (v1) \foreach \i in {2,...,\k} { -- (v\i) } -- cycle;
      \foreach \i in {1,...,\k} {
        \fill (v\i) circle (1pt);
      }

      \draw (v1) node[right]{\scriptsize \underline{10}0};
      \draw (v2) node[right]{\scriptsize \underline{10}1};
      \draw (v3) node[left]{\scriptsize \underline{11}1};
      \draw (v4) node[left]{\scriptsize \underline{11}0};
      \draw (v5) node[left]{\scriptsize \underline{01}0};
      \draw (v6) node[left]{\scriptsize \underline{01}1};
      \draw (v7) node[right]{\scriptsize \underline{00}1};
      \draw (v8) node[right]{\scriptsize \underline{00}0};

    \end{tikzpicture}
    \quad
    \begin{tikzpicture}[scale=1.2]
      \def\k{8}
      \foreach \i in {1,...,\k} {
        \pgfmathsetmacro{\angle}{(2*\i-1)*180/\k}
        \pgfmathsetmacro{\radius}{1/cos(180/\k)}
        \coordinate (v\i) at ({\radius*cos(\angle)}, {\radius*sin(\angle)});
      }

      \fill[gray!50] (v5) -- (w5) -- (w4) -- cycle;
      \draw[thick] (w5) -- (w4);
      \fill[gray!50] (v8) -- (w8) -- (w7) -- cycle;
      \draw[thick] (w8) -- (w7);

      \draw[thick] (v1) \foreach \i in {2,...,\k} { -- (v\i) } -- cycle;
      \foreach \i in {1,...,\k} {
        \fill (v\i) circle (1pt);
      }

      \draw (v1) node[right]{\scriptsize \underline{1}0\underline{0}};
      \draw (v2) node[right]{\scriptsize \underline{1}0\underline{1}};
      \draw (v3) node[left]{\scriptsize \underline{1}1\underline{1}};
      \draw (v4) node[left]{\scriptsize \underline{1}1\underline{0}};
      \draw (v5) node[left]{\scriptsize \underline{0}1\underline{0}};
      \draw (v6) node[left]{\scriptsize \underline{0}1\underline{1}};
      \draw (v7) node[right]{\scriptsize \underline{0}0\underline{1}};
      \draw (v8) node[right]{\scriptsize \underline{0}0\underline{0}};
    \end{tikzpicture}
    \quad
    \begin{tikzpicture}[scale=1.2]
      \def\k{8}
      \foreach \i in {1,...,\k} {
        \pgfmathsetmacro{\angle}{(2*\i-1)*180/\k}
        \pgfmathsetmacro{\radius}{1/cos(180/\k)}
        \coordinate (v\i) at ({\radius*cos(\angle)}, {\radius*sin(\angle)});
      }

      \fill[gray!50] (v5) -- (w5) -- (w4) -- cycle;
      \draw[thick] (w5) -- (w4);
      \fill[gray!50] (v7) -- (w7) -- (w6) -- cycle;
      \draw[thick] (w7) -- (w6);
      \fill[gray!50] (v8) -- (w8) -- (w7) -- cycle;
      \draw[thick] (w8) -- (w7);

      \draw[thick] (v1) \foreach \i in {2,...,\k} { -- (v\i) } -- cycle;
      \foreach \i in {1,...,\k} {
        \fill (v\i) circle (1pt);
      }
      \draw (v1) node[right]{\scriptsize 100};
      \draw (v2) node[right]{\scriptsize 101};
      \draw (v3) node[left]{\scriptsize 111};
      \draw (v4) node[left]{\scriptsize 110};
      \draw (v5) node[left]{\scriptsize 010};
      \draw (v6) node[left]{\scriptsize 011};
      \draw (v7) node[right]{\scriptsize 001};
      \draw (v8) node[right]{\scriptsize 000};
    \end{tikzpicture}
  \end{center}
  \caption{Illustration of Lemma~\ref{lem:main} for~$V=\{1,2,3\}$ and
    $E=\{\{1,2\},\{1,3\}\}$. From left to right, the polytopes~$\bar
    P_n\cap\bar P_n^{\{1,2\}}$, $\bar P_n\cap\bar P_n^{\{1,3\}}$, and
    $\bar P_G$ are depicted, where the shaded areas mark the
    triangles~$\Delta_\ell$ not belonging to the respective
    polytope. In~$\bar P_G$, the latter triangles correspond to the
    three dependent sets~$\{1,2\}$, $\{1,3\}$, and~$\{1,2,3\}$
    in~$G$.}\label{fig:main}
\end{figure}
For the following, let
\[
\delta_n:=\tfrac 12\tan(\tfrac\pi{2^n})(1-\cos(\tfrac\pi{2^{n-1}}))
\]
denote the area of each triangle~$\Delta_\ell$. It is now easy to
derive
\begin{lemma}\label{lem:vol}
  The number of independent sets in~$G$ is $2^n-\tfrac
  1{\delta_n}(\vol(Q_{2^n})-\vol(\bar P_G))$.
\end{lemma}
\begin{proof}
  The interiors of all triangles~$\Delta_\ell$ are pairwise disjoint
  and have the same area $\delta_n$. Moreover, by
  Lemma~\ref{lem:main}, the triangle $\Delta_\ell$ is missing in~$\bar
  P_G$ if and only if there exists an edge~$e$ with~$e\subseteq
  S_\ell$, i.e., if and only if~$S_\ell$ is a dependent set in~$G$. We
  derive that $\vol(\bar
  P_G)=\vol(Q_{2^{n}})-(2^n-n_{\text{IS}})\delta_n$, where
  $n_{\text{IS}}$ denotes the number of independent sets in~$G$. This
  yields the desired result.
\end{proof}
By Lemma~\ref{lem:vol}, the number of independent sets in a graph~$G$
can be derived from the area of~$\bar P_G$. In order to turn this
connection into a polynomial reduction, we next approximate all
irrational coefficients in the construction of~$\bar P_G$. In view of
Lemma~\ref{lem:vol}, it suffices to ensure that the resulting error
in~$\vol(\bar P_G)$ is bounded by~$\tfrac 14\delta_n$.
\begin{lemma}\label{lem:prec}
  For~$n\ge 2$, we have $\delta_n>2^{-3n+3}$.
\end{lemma}
\begin{proof}
  Since~$\tfrac\pi{2^{n-1}}\in [0,\tfrac\pi 2]$, we
  obtain~$\tan(\tfrac\pi{2^n})\ge\tfrac 1{2^{n-1}}$
  and~$1-\cos(\tfrac\pi{2^{n-1}})> \tfrac 1{2^{2n-3}}$ and thus the
  desired lower bound.
\end{proof}
\noindent Now let~$\varepsilon_n:=2^{-5n-1}$ and choose~$t_i\in\Q$
with $|t_i - \tan(\tfrac\pi{2^{i}})|\le\varepsilon_n$ for
all~$i=2,\dots,n+1$. Defining~$c_i:=(1-t_i^2)/(1+t_i^2)$
and~$s_i:=(2t_i)/(1+t_i^2)$, the rational matrix
\[
R_i^\circ:=\begin{pmatrix} c_i & s_i\\ -s_i & c_i\end{pmatrix}
\]
then describes the clockwise rotation by~$2\arctan(t_i)$
radians. Extending this definition by~$c_1:=-1$ and~$s_1:=0$, we have
\[
|c_i-\cos(\tfrac\pi{2^{i-1}})|,\;|s_i-\sin(\tfrac\pi{2^{i-1}})|,\;\|R^\circ_i-R_i\|_2\le
2\varepsilon_n
\]
for all~$i=1,\dots n+1$.

Now let~$P_n^{\circ}$ be the rational polyhedron defined exactly
as~$P_n$, except that all coefficients~$\cos(\tfrac\pi{2^{i-1}})$
and~$\sin(\tfrac\pi{2^{i-1}})$ are replaced by~$c_i$ and~$s_i$,
respectively. For~$e\in E$, define~$P_n^{e\circ}$ from~$P_n^e$
analogously. For simplifying proofs, we add the four rational
constraints~$-2\le y_{n+1},z_{n+1} \le 2$ to~$P_n^\circ$ and
all~$P_n^{e\circ}$. Finally, let~$\bar P_n^{\circ}$ and~$\bar
P_n^{e\circ}$ be the projections of~$P_n^{\circ}$ and~$P_n^{e\circ}$
to~$(y_1,z_1)$ and let~$\bar P_G^\circ$ be the intersection of all
these projections. Since the number of binary digits
in~$\varepsilon_n$ is polynomial in~$n$, an extended formulation
of~$\bar P_G^\circ$ can be constructed in polynomial time from the
input graph~$G$.
\begin{proposition}\label{prop:approx}
  $|\vol(\bar P_G)-\vol(\bar P_G^\circ)|\le \tfrac 14\delta_n$
\end{proposition}
The proof is reported in Appendix~\ref{app_B}.
\begin{theorem}\label{theo:vol}
  Given a graph~$G$, it is \card P-hard to compute $\vol(\bar P_G^\circ)$ up to an
  additive error of less than $\tfrac 14\delta_n$.
\end{theorem}
\begin{proof}
  The result immediately follows from Lemma~\ref{lem:vol},
  Proposition~\ref{prop:approx}, and the fact that it is \card P-hard to
  compute~$n_{\text{IS}}$~\cite[Theorem~4.2]{Dyer00}.
\end{proof}
We emphasize that the proof of Theorem~\ref{theo:vol} does \emph{not}
show strong \card P-hardness, although \card IS is strongly \card
P-hard: In order to obtain a close enough approximation of the
irrational polytope~$\bar P_G$, we need coefficients with a linear
number of digits, hence the numbers constructed in the reduction grow exponentially.
\begin{theorem}\label{theo:small_d}
  Problem~\eqref{prob:sslp} is \card P-hard for~$d=1$ and~$n_1=1$.
\end{theorem}
\begin{proof}
  Since~$\bar P_G^\circ$ is given by a compact extended formulation,
  its projection to the first coordinate~$y_1$ is an interval~$[l,u]$
  such that~$l$ and~$u$ can be computed in polynomial time by linear
  optimization. We have
  \begin{eqnarray*}
    \vol(\bar P_G^\circ)
    & = & \textstyle \int_{l}^u (\max\{z_1\mid (y_1,z_1)\in \bar P_G^\circ\}-\min\{z_1\mid (y_1,z_1)\in \bar P_G^\circ\}) \d y_1\\
    & = & (u-l)(\E_\xi[\max\{z_1\mid (\xi,z_1)\in \bar P_G^\circ\}]-\E_\xi[\min\{z_1\mid (\xi,z_1)\in \bar P_G^\circ\}])
  \end{eqnarray*}
  with~$\xi\sim{\cal U}[l,u]$. By Theorem~\ref{theo:vol}, it is thus
  \card P-hard to compute both these expected values. It thus remains
  to show that the computation of $\E_\xi[\min\{z_1\mid (\xi,z_1)\in
    \bar P_G^\circ\}]$ can be polynomially reduced
  to~\eqref{prob:sslp} with~$n_1=1$; the max-case is analogous.  For
  this, similar to~\cite{Hanasusanto16}, we consider the problem of
  minimizing
  \[
  -cx+\E_\xi[\min\{z_1\mid (x\xi,z_1)\in x\bar
    P_G^\circ\}]=x(\E_\xi[\min\{z_1\mid (\xi,z_1)\in \bar P_G^\circ\}]-c)
  \]
  over the single first-stage variable~$x\in [0,1]$. This is a problem
  of type~\eqref{prob:sslp}. Indeed, an extended formulation of~$x\bar
  P_G^\circ$ is obtained by multiplying all right-hand side vectors in
  the construction of~$P_G^\circ$ with~$x$, so that the former right
  hand side now takes the role of the transport matrix~$T_\xi$,
  together with the term~$x\xi$ replacing the former
  variable~$y_1$. The latter is the only term containing the random
  variable~$\xi$.
  
  For any~$c\in\Q$, if the optimum of the constructed problem is
  attained at~$x=0$, we derive~$\E_\xi[\min\{z_1\mid (\xi,z_1)\in \bar
    P_G^\circ\}]\ge c$, otherwise~$\E_\xi[\min\{z_1\mid (\xi,z_1)\in
    \bar P_G^\circ\}]\le c$; this is also true if the optimal~$x$ is
  not unique. We can thus use a standard bisection approach for~$c$,
  where the relevant range is~$[-16,0]$, since~$\bar
  P_G^\circ\in[-2,2]$. By Lemma~\ref{lem:prec}, the necessary
  precision~$\tfrac 18\delta_n$ is achieved after at
  most~$\log_2(16/(\tfrac 18\delta_n))\le 3n+4$ calls
  to~\eqref{prob:sslp}, concluding the proof.
\end{proof}

\section{Fixed size of the second-stage problem}\label{sec:fixed_mn}

We next prove that bounding the number of variables and constraints in
the second stage is not sufficient to obtain a tractable special case
of~\eqref{prob:sslp} either. This can be shown by similar techniques
as employed in~\cite{Hanasusanto16}, although in the reduction
presented there, the size of the second stage is not bounded and only
weak \card P-hardness is shown for the case of fixed recourse. We
first fix only the number of constraints~$m_2$.
\begin{theorem}\label{theo:small_m2}
  Problem~\eqref{prob:sslp} is strongly \card P-hard if~$m_2=1$, even
  if only the objective~$q_\xi$ is random.
\end{theorem}
\begin{proof}
  For a given poset~$\cal P$, consider the order polytope
  \[
  O({\cal P}):=\{\xi\in[0,1]^d\mid x_i\le x_j\text{ if }i\prec_{\cal
    P}j\}=\{\xi\in\R^d\mid A\xi\le b\}\;,
  \]
  where~$A\in\{-1,0,1\}^{m\times d}$ and~$b\in\{0,1\}^m$ collect all
  constraints defining~$O({\cal P})$; let~$a_1,\dots,a_m$ be the rows
  of~$A$. It is well-known that~$\vol(O({\cal P}))$, up to a factor
  of~$1/d\,!$, agrees with the number of linear extensions of~$\cal
  P$~\cite{Stanley86}. In particular, it is strongly \card P-hard to
  compute~$\vol(O({\cal P}))$ up to a precision of less
  than~$1/(2d\,!)$. But we can efficiently compute one linear
  extension~$\pi$. Then~$c_{\pi(i)}:=\tfrac i{d+1}$ defines an
  interior point~$c$ of~$O({\cal P})$ and we can rewrite
  \[
  O({\cal P})=\{\xi\in\R^n\mid\tilde A(\xi-c)\le e\}\;,
  \]
  where~$e=(1,\dots,1)^\top\in\R^m$ and~$\tilde A$ consists of the
  rows~$\tilde a_i:=1/(b_i-a_ic)\cdot a_i$ for
  all~$i=1,\dots,m$. Defining
  \[
  P_t:=c+(1-t)(O({\cal P})-c)=\{\xi\in\R^n\mid \tilde A(\xi-c)\le
  (1-t)e\},\quad t\in[0,1]\;,
  \]
  we have~$\vol (P_t)=(1-t)^d\vol(O({\cal P}))$. For~$x\in[0,1]$
  and~$t\in [0,1]$, we now consider the following second-stage problem
  \begin{align*}
    Q_{t}(x,\xi) \; = \; \min \quad & ((1-t)e-\tilde A(\xi-c))^\top y\\
    \st \quad & y\ge 0,\quad e^\top y\le x\;.
  \end{align*}
  Introducing one slack variable for the last inequality, we obtain a
  problem in standard form with~$m_2=1$ and~$n_2=m+1$ where all
  coefficients -- expect possibly~$t$ -- have polynomial
  values. Now~$Q_{t}(x,\xi)=\min\,\{0,\min\,\{1-t-\tilde a_j
  (\xi-c)\mid j=1,\dots,m\}\}\cdot x$ and hence
  \[
  \tfrac\partial{\partial t} Q_{t}(x,\xi)=\left\{\begin{array}{cl}
  0 & \text{if }\xi \in\operatorname{int} P_t \\
  -x & \text{if }\xi \not\in P_t\;.
  \end{array}\right.
  \]
  Using the Leibniz integral rule and assuming~$\xi\sim{\cal
    U}[0,1]^d$, we obtain
  \[
  \tfrac\partial{\partial t} \E_\xi[Q_{t}(x,\xi)]
  = \E_\xi[\tfrac\partial{\partial t} Q_{t}(x,\xi)]
  = (\vol(P_t)-1)\cdot x=((1-t)^d\vol(O({\cal P}))-1)\cdot x\;.
  \]
  Integration over~$t\in[0,1]$ yields
  \[
  \vol(O({\cal P}))=(d+1)(1+\tfrac
  1x(\E_\xi[Q_{1}(x,\xi)]-\E_\xi[Q_{0}(x,\xi)]))\;.
  \]
  Altogether, for~$x>0$, the computation of~$\vol(O({\cal P}))$ can be
  polynomially reduced to the computation of~$\E_\xi[Q_{t}(x,\xi)]$
  for~$t\in\{0,1\}$ and thus it is strongly \card P-hard to compute
  the latter up to an accuracy of less than~$1/(4(d+1)d\,!)$.

  To conclude the proof, it remains to reduce the computation of the
  expected values~$F_t:=\E_\xi[Q_{t}(1,\xi)]\in[-d-1,0]$,
  $t\in\{0,1\}$, to the optimization problem~\eqref{prob:sslp}. As in
  the proof of Theorem~\ref{theo:small_d}, we can achieve this by
  bisection, solving~\eqref{prob:sslp} for different values
  of~$c\in[-d-1,0]$. As argued above, it suffices to determine~$F_t$
  up to a precision of~$1/(4(d+1)d\,!)$, so~$K:=2+\lceil
  \log((d+1)^2d\,!)\rceil$ is a polynomial upper bound on the number
  of bisection steps needed.

  Each relevant number~$c\in[-d-1,0]$ can thus be written
  as~$-(d+1)\sum_{i=1}^Kc_i2^{-i}$ with~$c\in\{0,1\}^K$. In the first
  stage of~\eqref{prob:sslp}, we now introduce
  variables~$x_1,\dots,x_K$ along with constraints~$2x_{i+1}=x_i$
  for~$i=1,\dots,K-1$ and a first-stage objective
  function~$(d+1)\sum_{i=1}^Kc_ix_i$, which agrees with~$-cx_1$. Then
  all coefficients in the construction are polynomial, and by
  minimizing~$\sum_{i=1}^K(d+1)c_ix_i+F_tx_1$ over~$x_1\in[0,1]$ we
  can decide whether~$F_t\ge c$. This shows that~\eqref{prob:sslp} is
  strongly \card P-hard.
\end{proof}
Note that the proof of Theorem~\ref{theo:small_m2} requires a
non-constant number of variables in the first stage in order to avoid
exponentially small coefficients in the first-stage objective, which
would prevent showing \card P-hardness in the strong sense.

We emphasize that \emph{weak} \card P-hardness of~\eqref{prob:sslp}
for~$m_2=1$ also follows easily from Corollary~1
in~\cite{Hanasusanto16}. The latter result actually implies that \card
P-hardness persists even when fixing~$n_2$ as well.
\begin{theorem}[{\cite{Hanasusanto16}}]\label{theo:small_n2}
  Problem~\eqref{prob:sslp} is \card P-hard
  if~$n_1=1$,~$n_2=2$,~$m_2=1$, and only the objective~$q_\xi$ is
  random.
\end{theorem}
Comparing the results of Theorem~\ref{theo:small_m2} and
Theorem~\ref{theo:small_n2}, it remains to clarify whether
fixing~$n_2$ (and thus also~$m_2$) leads to a strongly or weakly \card
P-hard problem. The latter is the case.
\begin{theorem}\label{theo:pseudo}
  Let~$n_2$ be fixed. Then~\eqref{prob:sslp} can be solved up to any
  desired accuracy~$\varepsilon>0$ in pseudopolynomial time (and in
  time logarithmic in~$1/\varepsilon$).
\end{theorem}
\begin{proof}
  It is easy to verify that $\E_\xi[Q(x,\xi)]$ is a convex function
  in~$x$, thus by Grötschel et al.~\cite[Theorem~4.3.13]{Grotschel93}
  -- and due to our assumption that the first-stage feasible region is
  non-empty and bounded -- it suffices to show that~$\E_\xi[Q(x,\xi)]$
  can be computed in pseudopolynomial time for any given~$x\in\R^{n_1}$.

  Let~$B$ be the set of bases of~$W$. Since~$n_2$ is fixed and~$m_2\le
  n_2$, also~$|B|\le n_2^{m_2}$ is bounded. For each~$\B\in B$,
  let~$P_\B$ denote the set of all~$\xi\in[l,u]^d$ such that~$\B$ is
  feasible for the given~$x$ and has non-negative reduced costs, i.e.,
  \[
  P_\B:=\{\xi\in [l,u]^d\mid W_\B^{-1}(h_\xi-T_\xi x)\ge
  0,~(q_\xi)_\NB-W_\NB^\top (W_\B^{-1})^\top(q_\xi)_\B\ge 0\}\;.
  \]
  Each~$P_\B$ is a polytope defined by intersecting
  the box~$[l,u]^d$ with~$n_2$ many halfspaces.
  
  Next, we partition~$B$ into equivalence
  classes~$B^{(1)},\dots,B^{(\ell)}$ with respect to the quadratic
  function $(q_\xi)_{\B}^\top W_{\B}^{-1}(h_\xi-T_\xi x)$ that
  describes the objective value of the basic solution corresponding
  to~$\B$ as a function of~$\xi$. More precisely, two bases are
  equivalent if this function agrees on~$\R^d$; we denote the latter
  by~$q^{(r)}$ for the class~$B^{(r)}$. Two bases~$\B_1$ and~$\B_2$
  belonging to different classes are simultaneously optimal on a zero
  set, since otherwise they agree everywhere; hence~$\vol(P_{\B_1}\cap
  P_{\B_2})=0$. In particular, setting~$P^{(r)}:=\cup_{\B\in
    B^{(r)}}P_{\B}$, we obtain
  \[
  \E_\xi[Q(x,\xi)] = \E_\xi\big[\min_{\B\in B}\;(q_\xi)_\B^\top W_\B^{-1}(h_\xi-T_\xi x)\big]
  = \Big(\prod_{i=1}^d\tfrac 1{u_i-l_i}\Big)\sum_{r=1}^\ell \int_{P^{(r)}} q^{(r)}(\xi) \d\xi\;.
  \]
  By the inclusion-exclusion principle, we have
  \begin{equation}
    \int_{P^{(r)}} q^{(r)}(\xi) \d\xi
    = \sum_{\emptyset\neq B'\subseteq B^{(r)}}(-1)^{|B'|+1}
    \int_{\cap_{\B\in B'} P_{\B}} q^{(r)}(\xi)\d\xi\;.\label{eq:incl_excl}\tag{*}
  \end{equation}\\[-1.5ex]
  Since~$|B^{(r)}|\le |B|\le n_2^{m_2}$, this sum ranges over at
  most~$2^{n_2^{m_2}}-1$ many subsets and each
  intersection~$\cap_{\B\in B'} P_{\B}$ is a polytope defined by box
  constraints and at most~$n_2^{m_2+1}$ linear inequalities. All
  integrals in~\eqref{eq:incl_excl} can thus be computed in
  pseudopolynomial time by Lemma~\ref{lem:quad} in
  Appendix~\ref{app}. This yields the desired result.
\end{proof}

\section{Fixing both sizes}\label{sec:fixed_both}

We finally discuss the case where both the dimension~$d$ of the random
vector and the number~$m_2$ of second-stage constraints are fixed.
\begin{theorem}\label{theo:both}
  Let~$d$ and~$m_2$ be fixed. Then~\eqref{prob:sslp} can be solved up
  to any desired accuracy~$\varepsilon>0$ in polynomial time (and in
  time logarithmic in~$1/\varepsilon$).
\end{theorem}
\begin{proof}
  The result can be shown analogously to
  Theorem~\ref{theo:pseudo}. Again, it suffices to
  compute~$\E_\xi[Q(x,\xi)]$, but the number of second-stage
  bases~$|B|$ is now polynomial instead of fixed. In particular, the
  number of non-zero terms in the sum in~\eqref{eq:incl_excl} can now
  become exponential. However, the total number of constraints
  defining all polytopes~$P_{\cal B}$ with~${\cal B}\in B^{(r)}$
  is~$n_2^{m_2+1}$ and hence polynomial, as~$m_2$ is fixed. Since~$d$
  is fixed as well, we can compute the corresponding hyperplane
  arrangement and identify all cells belonging to~$P^{(r)}$ in
  polynomial time; see, e.g.,~\cite{Edelsbrunner86}. Finally,
  integrating~$q^{(r)}$ over~$P^{(r)}$ reduces to
  integrating~$q^{(r)}$ over these cells, which can be done
  efficiently for fixed~$d$ using triangulation; see,
  e.g.,~\cite{Lawrence91}.
\end{proof}

\section{Conclusion}

The following table summarizes all results. The complexities stated
here refer to both the evaluation of the objective function and the
optimization in~\eqref{prob:sslp}. Recall that~$m_2\le n_2$ and hence
fixed~$n_2$ implies fixed~$m_2$.
\bigskip
\begin{center}
\begin{tabular}{l|c|c|c|c}
  fixed parameter & $d$ & $m_2$ & $n_2$ & $d$ and $m_2$ \\
  \hline
  complexity & \card P-hard & strongly & weakly & polynomial\\
  & & \card P-hard & \card P-hard\\
  \hline
  references & Theorem~\ref{theo:small_d} & Theorem~\ref{theo:small_m2} & Theorem~\ref{theo:small_n2} & Theorem~\ref{theo:both}\\
  & & & Theorem~\ref{theo:pseudo}
\end{tabular}
\end{center}
\bigskip
\noindent An obvious open question is whether the case~$d=1$ is
actually strongly \card P-hard or, if it is not, whether the problem
starts to be strongly \card P-hard for some constant~$d$. In fact, one
can construct polytopes in~$\R^n$ with bounded integer coefficients
having two-dimensional projections with~$2^n$ vertices, e.g., based on
the Goldfarb cube~\cite{Goldfarb94}. However, these projections are
much less regular than the $2^n$-gon used in our construction, so that
it is not clear how to reduce any \card P-hard counting problem to the
computation of the volume of a polytope derived from such a
projection.

\bibliographystyle{plain}
\bibliography{stoch}

@article{BenTal01,
  author =	 {Aharon Ben-Tal and Arkadi Nemirovski},
  journal =	 {Mathematics of Operations Research},
  number =	 2,
  pages =	 {193--205},
  publisher =	 {INFORMS},
  title =	 {On Polyhedral Approximations of the Second-Order
                  Cone},
  volume =	 26,
  year =	 2001,
  doi =		 {10.1287/moor.26.2.193.10561}
}

@book{Birge11,
  title =	 {Introduction to Stochastic Programming},
  journal =	 {Springer Series in Operations Research and Financial
                  Engineering},
  publisher =	 {Springer New York},
  author =	 {Birge, John R. and Louveaux, Fran\c{c}ois},
  year =	 2011,
  doi =		 {10.1007/978-1-4614-0237-4}
}

@article{Dyer00,
  title =	 {On {M}arkov Chains for Independent Sets},
  volume =	 35,
  DOI =		 {10.1006/jagm.1999.1071},
  number =	 1,
  journal =	 {Journal of Algorithms},
  publisher =	 {Elsevier BV},
  author =	 {Dyer, Martin and Greenhill, Catherine},
  year =	 2000,
  pages =	 {17--49}
}

@article{Edelsbrunner86,
  title =	 {Constructing Arrangements of Lines and Hyperplanes
                  with Applications},
  volume =	 15,
  DOI =		 {10.1137/0215024},
  number =	 2,
  journal =	 {SIAM Journal on Computing},
  author =	 {Edelsbrunner, Herbert and O’Rourke, Joseph and
                  Seidel, Raimund},
  year =	 1986,
  pages =	 {341--363}
}

@book{Grotschel93,
  author =	 "Gr{\"o}tschel, Martin and Lov{\'a}sz, L{\'a}szl{\'o}
                  and Schrijver, Alexander",
  title =	 "Geometric Algorithms and Combinatorial Optimization",
  publisher =	 "Springer Berlin",
  year =	 1993,
  series =	 "Algorithms and Combinatorics",
  address =	 "Heidelberg",
  doi =		 "10.1007/978-3-642-78240-4"
}

@article{Hanasusanto16,
  title =	 {A comment on “computational complexity of stochastic
                  programming problems”},
  author =	 {Hanasusanto, Grani A. and Kuhn, Daniel and
                  Wiesemann, Wolfram},
  journal =	 {Mathematical Programming},
  volume =	 159,
  pages =	 {557--569},
  year =	 2016,
  doi =		 {10.1007/s10107-015-0958-2}
}

@Article{Lawrence91,
  author =	 {Jim Lawrence},
  title =	 {Polytope Volume Computation},
  journal =	 {Mathematics of Computation},
  year =	 1991,
  volume =	 57,
  number =	 195,
  pages =	 {259--271},
  doi =		 {10.2307/2938672}
}

@book{deBoor93,
  title =	 {Box Splines},
  series =	 {Applied Mathematical Sciences},
  number =	 98,
  publisher =	 {Springer New York},
  author =	 {de Boor, Carl and H\"{o}llig, Klaus and
                  Riemenschneider, Sherman},
  year =	 1993,
  doi =		 {10.1007/978-1-4757-2244-4}
}

@article{Forcier24,
  title =	 {Exact Quantization of Multistage Stochastic Linear
                  Problems},
  volume =	 34,
  DOI =		 {10.1137/22m1508005},
  number =	 1,
  journal =	 {SIAM Journal on Optimization},
  publisher =	 {Society for Industrial & Applied Mathematics (SIAM)},
  author =	 {Forcier, Maël and Gaubert, Stéphane and Leclère,
                  Vincent},
  year =	 2024,
  pages =	 {533--562}
}

@inbook{Goldfarb94,
  title =	 {On the Complexity of the Simplex Method},
  DOI =		 {10.1007/978-94-015-8330-5_2},
  booktitle =	 {Advances in Optimization and Numerical Analysis},
  publisher =	 {Springer Netherlands},
  author =	 {Goldfarb, Donald},
  year =	 1994,
  pages =	 {25--38}
}

@article{Stanley86,
  title =	 {Two poset polytopes},
  volume =	 1,
  DOI =		 {10.1007/bf02187680},
  number =	 1,
  journal =	 {Discrete \& Computational Geometry},
  author =	 {Stanley, Richard P.},
  year =	 1986,
  pages =	 {9–23}
}

\appendix

\section{Bounding the area error}\label{app_B}
The aim of this section is to show Proposition~\ref{prop:approx}. For
this, we will use the following auxiliary result.
\begin{lemma}\label{lem:prop}
  Let~$a_1,a_1^\circ,a_2\in\R^2$
  with~$\|a_1-a_1^\circ\|_2\le\varepsilon'$ and $b\in\{0,1\}^n$.  Then
  \[
  |(a_1^\circ)^\top U^\circ a_2- a_1^\top Ua_2| \le
  (\varepsilon'+2(n-1)\varepsilon_n\|a_1\|_2)\|a_2\|_2\;,
  \]
  where~$U := R^{b_n}R_n\cdots R^{b_1}R_1$ and~$U^\circ :=
  R^{b_n}R_n^\circ\cdots R^{b_1}R_1^\circ$.
\end{lemma}
\begin{proof}
  Since $\|R_i\|_2=\|R_i^\circ\|_2=1$, $\|R_i-R^\circ_i\|_2\le
  2\varepsilon_n$ for all~$i=2,\dots,n$, and finally~$R_1=R_1^\circ$,
  a telescoping argument yields $\|U^\circ-U\|_2\le
  2(n-1)\varepsilon_n$. Using that $(a_1^\circ)^\top U^\circ a_2-
  a_1^\top Ua_2 = (a_1^\circ-a_1)^\top U^\circ a_2 + a_1^\top (U^\circ
  -U) a_2$, the result follows.
\end{proof}
For the following proofs, set~$\rho_n:=8(n+1)\varepsilon_n$.
\begin{lemma}\label{lem:inc1}
  $\bar P_G\subseteq (1+\rho_n)\bar P_G^\circ$.
\end{lemma}
\begin{proof}
  By Lemma~\ref{lem:main}, the projected polytope~$\bar P_G$ agrees
  with the convex hull of all~$2^n$ vectors~$w_\ell$ and of those
  vectors~$v_\ell$ corresponding to independent sets in~$G$. First
  consider a vector~$w_\ell$ and define~$(y,z)$ recursively by
  setting~$(y_1,z_1)^\top:=w_\ell$ as well
  as~$(y_{i+1},z_{i+1}):=R_i^\circ(y_i,z_i)^\top$ if the second entry
  of~$R_i^\circ(y_i,z_i)^\top$ is non-negative
  and~$(y_{i+1},z_{i+1}):=RR_i^\circ(y_i,z_i)^\top$ otherwise;
  let~$b\in\{0,1\}^n$ be the corresponding binary
  vector. For~$a=(\cos(\tfrac\pi{2^{n-1}}),\sin(\tfrac\pi{2^{n-1}}))^\top$
  and~$a^\circ=(c_{n},s_{n})^\top$, we have~$\|a-a^\circ\|_2\le
  2\varepsilon_n$, so Lemma~\ref{lem:prop} yields
  \begin{eqnarray*}
    (a^\circ)^\top(y_{n+1},z_{n+1})^\top
    & = & (a^\circ)^\top R^{b_n}R_n^\circ\cdots R^{b_1}R_1^\circ w_\ell \\
    & \le & a^\top R^{b_n}R_n\cdots R^{b_1}R_1w_\ell+2n\varepsilon_n\quad\text{(as $\|w_\ell\|_2=1$)}\\
    & \le & 1+2n\varepsilon_n\quad\text{(as $R^{b_n}R_n\cdots R^{b_1}R_1w_\ell\in P_n$)}
  \end{eqnarray*}
  and
  \begin{eqnarray*}
    y_{n+1}
    & = & (1,0)^\top R^{b_n}R_n^\circ\cdots R^{b_1}R_1^\circ w_\ell \\
    & \le & (1,0)^\top R^{b_n}R_n\cdots R^{b_1}R_1w_\ell+2(n-1)\varepsilon_n\quad\text{(as $\|w_\ell\|_2=\|(1,0)^\top\|_2=1$)}\\
    & \le & 1+2(n-1)\varepsilon_n\quad\text{(as $R^{b_n}R_n\cdots R^{b_1}R_1w_\ell\in P_n$)}\;.
  \end{eqnarray*}
  Since all recursive constraints in the definition of~$P_n^{\circ}$
  hold by construction of~$(y,z)$, we derive
  that~$(y,z)\in(1+2n\varepsilon_n)P_n^{\circ}$ and
  hence~$w_\ell=(y_1,z_1)\in(1+2n\varepsilon_n)\bar
  P_n^{\circ}$. Moreover, $\|(y_{n+1},z_{n+1})\|_2=1$ and
  hence~$(y_{n+1},z_{n+1})\in(1+\rho_n)[-2,2]^2$.

  Similarly, for~$e\in E$, define~$(y,z)$ from~$(y_1,z_1)=w_\ell$ such
  that all recursive constraints of~$P_n^{e\circ}$ are satisfied. Then
  for~$a=(\cos(\tfrac\pi{2^{n}}),\sin(\tfrac\pi{2^{n}}))^\top$
  and~$a^\circ=(c_{n+1},s_{n+1})^\top$,
  \begin{eqnarray*}
    (a^\circ)^\top(y_{n+1},z_{n+1})^\top
    & = & (a^\circ)^\top R^{b_n}R_n^\circ\cdots R^{b_1}R_1^\circ w_\ell \\
    & \le & a^\top R^{b_n}R_n\cdots R^{b_1}R_1w_\ell+2n\varepsilon_n\quad\text{(as $\|w_\ell\|_2=1$)}\\
    & \le & \cos(\tfrac\pi{2^{n}})+2n\varepsilon_n\quad\text{(as $R^{b_n}R_n\cdots R^{b_1}R_1w_\ell\in P_n^e$)}\\
    & \le & c_{n+1}+2(n+1)\varepsilon_n\quad\text{(as $|c_{n+1}-\cos(\tfrac\pi{2^{n}})|\le 2\varepsilon_n$)}\\
    & \le & c_{n+1}(1+4(n+1)\varepsilon_n)
  \end{eqnarray*}
  since~$c_{n+1}\ge c_3\ge \cos(\tfrac\pi 4)-2\varepsilon_n\ge \tfrac
  12$ for~$n\ge 2$.  Hence~$w_\ell\in(1+4(n+1)\varepsilon_n)\bar
  P_n^{e\circ}$. It then follows that~$w_\ell\in (1+\rho_n)\bar
  P^\circ_G$.

  If~$v_\ell\in \bar P_G$, one can show analogously that~$v_\ell\in
  (1+2\sqrt 2n\varepsilon_n)\bar P_n^{\circ}$ since~$\|v_\ell\|_2\le
  \sqrt 2$.  For~$e\in E$, again define~$(y,z)$ recursively as before,
  so in particular, $b_i=0$ for~$i\in e$. Since~$v_\ell\in\bar P_n^e$,
  we have that~$e\not\subseteq S_\ell$ and thus~$b^\ell_i=1$ for
  some~$i\in e$. Hence~$b\neq b^\ell$ and thus~$R^{b_n}R_n\cdots
  R^{b_1}R_1v_\ell\neq v_0$. We conclude that
  \[
  (\cos(\tfrac\pi{2^{n}}),\sin(\tfrac\pi{2^{n}}))^\top
  R^{b_n}R_n\cdots R^{b_1}R_1v_\ell \le \cos(\tfrac\pi{2^{n}})
  \]
  so that~$v_\ell\in(1+(4\sqrt 2n+4)\varepsilon_n)\bar P_n^{e\circ}$
  follows as above using~$\|v_\ell\|_2\le \sqrt 2$. This finally
  implies~$v_\ell\in(1+\rho_n)\bar P_G^{\circ}$ as well.
\end{proof}
\begin{lemma}\label{lem:inc2}
  $\bar P_G^\circ\subseteq (1+\rho_n)\bar P_G$.
\end{lemma}
\begin{proof}
  Recall from Proposition~\ref{prop:proj} and Lemma~\ref{lem:main}
  that~$\bar P_G$ is defined by the inequalites~$w_\ell^\top
  (y_1,z_1)\le 1$ for all~$\ell$ and~$v_\ell^\top(y_1,z_1)\le 1$ for
  those~$\ell$ corresponding to dependent sets in~$G$.
  
  So let~$(y_1,z_1)\in \bar P_G^{\circ}$ and choose any
  extension~$(y,z)\in P_n^{\circ}$ of~$(y_1,z_1)$. By the recursive
  constraints, we
  have~$\|(y_{i+1},z_{i+1})\|_2\ge\|(y_{i},z_{i})\|_2$. But~$(y_{n+1},z_{n+1})\in[-2,2]^2$,
  hence~$\|(y_i,z_i)\|_2\le \sqrt 8$ for all~$i=1,\dots,n+1$.
  
  Now let~$a$ be the exact extension of~$w_\ell$ as in the second part
  of the proof of Proposition~\ref{prop:proj}. Then since~$(a_i)_2\ge
  0$, we have $a_i^\top R^{b_i^\ell}R_i^\circ (y_i,z_i)^\top\le
  a_i^\top(y_{i+1},z_{i+1})$ for all~$i=2,\dots,n$ and thus
  \begin{eqnarray*}
    a_{i-1}^\top(y_i,z_i)^\top & = & a_i^\top R^{b_i^\ell}R_i(y_i,z_i)^\top\\
    & \le & a_i^\top R^{b_i^\ell}R_i^\circ(y_i,z_i)^\top+4\sqrt 2\varepsilon_n\\
    & \le & a_i^\top (y_{i+1},z_{i+1})^\top+4\sqrt 2\varepsilon_n
  \end{eqnarray*}
  since $\|a_i\|_2=1$, $\|R_i-R_i^\circ\|_2\le 2\varepsilon_n$,
  and~$\|(y_i,z_i)\|_2\le \sqrt 8$ as shown above.  Altogether,
  \begin{eqnarray*}
    w_\ell^\top(y_1,z_1)^\top
    & \le & w_0^\top (y_{n+1},z_{n+1})^\top+4\sqrt2 (n-1)\varepsilon_n\\
    & \le & (c_n,s_n)^\top (y_{n+1},z_{n+1})^\top+4\sqrt 2n\varepsilon_n\\
    & \le & 1+4\sqrt 2n\varepsilon_n
  \end{eqnarray*}
  for all~$\ell$. The last inequality follows from~$(y,z)\in
  P_n^{\circ}$.

  Moreover, if~$\ell$ corresponds to a dependent set in~$G$,
  choose~$e\subseteq S_\ell$. Then similarly as above,
  since~$(y_1,z_1)\in \bar P_n^{e\circ}$, there exists an appropriate
  extension~$(y,z)\in P_n^{e\circ}$ of~$(y_1,z_1)$ such that
  \begin{eqnarray*}
    v_\ell^\top(y_1,z_1)^\top
    & \le & v_0^\top (y_{n+1},z_{n+1})^\top+8(n-1)\varepsilon_n\quad\text{(as~$\|v_\ell\|_2\le\sqrt 2$)}\\
    & = & \sec(\tfrac\pi{2^n})(\cos(\tfrac\pi{2^n}),\sin(\tfrac\pi{2^n}))^\top (y_{n+1},z_{n+1})^\top+8(n-1)\varepsilon_n\\
    & \le & \sec(\tfrac\pi{2^n})(c_{n+1}+4\sqrt 2\varepsilon_n)+8(n-1)\varepsilon_n\quad\text{(as $(y,z)\in P_n^{e\circ}$)}\\
    & \le & 1+8(n+1)\varepsilon_n\;.
  \end{eqnarray*}
  Hence $(y_1,z_1)\in (1+\rho_n)\bar P_G$.
\end{proof}

Now Proposition~\ref{prop:approx} can be shown as follows: By
Lemma~\ref{lem:inc1} and Lemma~\ref{lem:inc2} and since both polytopes
are two-dimensional, we have
\[
\tfrac{1}{(1+\rho_n)^2}\vol(\bar P_G)\leq \vol(\bar P_G^\circ) \leq
(1+\rho_n)^2\vol(\bar P_G)\;.
\]
Since~$\vol(\bar P_G)\leq 4$, it follows that
\[
|\vol(\bar P_G^\circ)-\vol(\bar P_G)| \le 4((1+\rho_n)^2-1) \le 64(n+1)\varepsilon_n+256(n+1)^2\varepsilon_n^2\;,
\]
which is less than~$\tfrac14\delta_n$ for~$n\ge 4$ by
Lemma~\ref{lem:prec} and for~$n=2,3$ by direct verfication.

\section{Auxiliary results on polytope volumes}\label{app}
Let~$A\in\Z^{m\times d}$ and~$b\in\Q^m$ be given with~$A\neq 0$, and
define~$P:=\{\xi\in[0,1]^d\mid A\xi\le b\}$.
\begin{lemma}\label{lem:pseudo_vol}
  If~$m$ is fixed, then the volume of~$P$ can be computed in time
  polynomial in~$d$ and~$\|A\|_\infty$ and in the encoding length
  of~$b$.
\end{lemma}
\begin{proof}
  Let $a_1,\dots,a_d\in\mathbb Z^m$ be the columns of $A$ and~$
  V:=\operatorname{span}(a_1,\dots,a_d)\subseteq\mathbb R^m$
  with~$r:=\dim V$. After permuting the columns of $A$, we may assume
  that $a_1,\dots,a_r$ are linearly independent.  For $k=r,\dots,d$,
  define $\varphi_k\colon \mathbb R^k\to V$,
  $\varphi_k(t):=\sum_{i=1}^k a_i t_i$,
  and~$Z_k:=\varphi_k([0,1]^k)$. Moreover, define~$h_k:V\to\mathbb R$
  recursively by
  \[
  h_r(y):=\tfrac{1}{|\det \varphi_r|}\mathbf 1_{Z_r}(y),\quad
  h_k(y):=\int_0^1 h_{k-1}(y-ta_k)\d t \quad\text{for
    $k=r+1,\dots,d$}\;.
  \]
  Using induction and Fubini, one can show that, for every measurable
  set $E\subseteq V$ and every $k=r,\dots,d$, we have that $\int_E
  h_k(y)\d y=\vol_k(\varphi_k^{-1}(E)\cap[0,1]^k)$. In particular,
  for~$E:=\{y\in V:y\le b\}$, we obtain
  $\varphi_d^{-1}(E)\cap[0,1]^d=P$, and hence
  \[
    \vol(P)=\int_E h_d(y)\d y\;.
  \]
  Now for $k=r,\dots,d$, define $S_k:=\{\sum_{j\in J}a_j\mid
  J\subseteq\{1,\dots,k\}\}$. Moreover, let $\mathcal H_k$ be the set
  of all affine hyperplanes of the form
  $s+\operatorname{span}\{a_i\mid i\in I\}$, where $s\in S_k$,
  $I\subseteq\{1,\dots,k\}$, and $\dim\operatorname{span}\{a_i\mid
  i\in I\}=r-1$. Let $\mathcal A_k$ denote the arrangement in~$V$
  induced by $\mathcal H_k$ and~$\mathcal C_k$ the set of
  full-dimensional cells in~$\mathcal A_k$.

  By definition, the function~$h_k$ is the box spline associated
  with~$(a_1,\dots,a_k)$. It thus follows from standard results on box
  splines that~$h_k|_C$ is a polynomial of degree at most $k-r$ for
  each~$C\in{\mathcal C_k}$; see, e.g., \cite[Prop.~37]{deBoor93}. Our
  next goal is to bound the size of $\mathcal C_k$. Since
  \[S_k\subseteq[-k\|A\|_\infty,k\|A\|_\infty]^m\cap\mathbb Z^m\;,
  \]
  we have~$|S_k|\le(2k\|A\|_\infty+1)^m$. Furthermore, there are at
  most~$\binom{k}{r-1}$ possible sets~$I\subseteq[k]$ of
  cardinality~$r-1$. Therefore,
  \[
  |\mathcal H_k|\le(2k\|A\|_\infty+1)^m\textstyle\binom{k}{r-1}\;.
  \]
  As~$m$ is fixed and~$r\le m$, this is polynomial in $k$
  and~$\|A\|_\infty$. It follows that $\mathcal A_k$ has polynomially
  many cells and that the arrangement and all its cells can be
  computed in time polynomial in~$d$ and~$\|A\|_\infty$, using once
  again that~$\dim V$ is fixed~\cite{Edelsbrunner86}.

  On each cell~$C\in\mathcal C_k$, the polynomial~$h_k|_C$ is of
  degree at most~$k-r$ and hence determined by at most~$k':=\binom kr$
  many coefficients, which can be determined recursively. For~$k=r$,
  the function~$h_r$ is constantly~$\tfrac{1}{|\det \varphi_r|}$
  on~$\operatorname{int}(Z_r)$ and zero on every other
  cell. For~$k>r$, we can use interpolation. More specifically,
  for~$C\in\mathcal C_k$, choose~$k'$ rational
  points~$y\in\operatorname{int}(C)$ in general position and
  compute~$h_{k}(y)=\int_0^1 h_{k-1}(y-t a_k)\d t$. The latter can be
  done efficiently, using that both~$\mathcal H_{k-1}$ and the
  coefficients of~$h_{k-1}|_D$ for each~$D\in{\mathcal A}_{k-1}$ are
  known explicitly by the induction hypothesis. In summary, we obtain
  the coefficients of~$h_d|_C$ for all~$C\in\mathcal C_d$ in time
  polynomial in~$d$ and~$\|A\|_\infty$.

  It remains to compute~$\int_E h_d$ from this information. The
  function $h_d$ vanishes outside $Z_d$. Thus
  \[
  \int_E h_d(y)\d y = \sum_{C\in\mathcal C_d} \int_{C\cap E\cap Z_d}h_d(y)\d y\;.
  \]
  Each set $C\cap E\cap Z_d$ is a bounded rational polytope in the
  fixed-dimensional space~$V$ defined by polynomially many
  constraints. On each~$C\in\mathcal C_d$, the function $h_d|_C$ is a
  polynomial of degree at most $d-r$, so that all integrals in the sum
  above can be computed efficiently~\cite{Lawrence91}.
\end{proof}

\begin{lemma}\label{lem:quad}
  Let~$b\in\Q^m$ and let~$q\colon\R^d\to\R$ be a quadratic
  function. If~$m$ is fixed, then the integral~$\int_{P} q(\xi)\d\xi$
  can be computed in time polynomial in~$d$ and~$\|A\|_\infty$ and in
  the encoding lengths of~$b$ and the coefficients of~$q$.
\end{lemma}
\begin{proof}
  Using linearity, it suffices to compute all moments~$\int_{P}
  \xi_i\xi_j\d\xi$,~$\int_{P} \xi_i\d\xi$, and~$\int_{P} 1\d\xi$
  for~$i,j\in\{1,\dots,d\}$. But~$\int_{P} \xi_i\xi_j\d\xi=\vol(P')$
  for the polytope
  \[
  P'=\{(\xi,u,v)\in [0,1]^{d+2}\mid A\xi\le b,u\le\xi_i,v\le\xi_j\}\;,
  \]
  and the integrals~$\int_{P} \xi_i\d\xi$ can be treated analogously.
  The result now follows directly from Lemma~\ref{lem:pseudo_vol}.
\end{proof}

\end{document}